\documentclass{amsart}
\usepackage{amsmath,amsthm,amscd,amssymb,amsfonts, amsbsy}
\usepackage{latexsym}
\usepackage{txfonts}

\numberwithin{equation}{section}

\theoremstyle{plain}
\newtheorem{theorem}[equation]{Theorem}
\newtheorem{lemma}[equation]{Lemma}

\theoremstyle{definition}

\theoremstyle{remark}
\newtheorem{remark}[equation]{Remark}

\newcommand{\dist}{\operatorname{dist}}

\newcommand{\mysection}[1]{\section{#1}
\setcounter{equation}{0}}

\newcommand{\bR}{\mathbb R}

\newcommand\wto{\rightharpoonup}
\newcommand\rV{\mathring{V}}

\newcommand\cLt{{}^t\!\mathcal{L}}

\newcommand{\ip}[1]{\left\langle#1\right\rangle}
\newcommand{\set}[1]{\left\{#1\right\}}
\newcommand{\norm}[1]{\lVert#1\rVert}

\newcommand{\abs}[1]{\left\lvert#1\right\rvert}

\renewcommand{\epsilon}{\varepsilon}
\renewcommand{\vec}[1]{\boldsymbol{#1}}

\begin{document}
\title[Green's matrix in two dimensions]
{Green's matrices of second order elliptic systems with measurable coefficients in two dimensional domains}

\author[H. Dong]{Hongjie Dong}
\address[H. Dong]{Division of Applied Mathematics, Brown University, 
182 George Street, Providence, RI 02912, United States of America}
\email{hdong@brown.edu}
\thanks{Hongjie Dong was partially supported by the
National Science Foundation under agreement No. DMS-0111298
and a start-up funding from the Division of Applied Mathematics of
Brown University.}

\author[S. Kim]{Seick Kim}
\address[S. Kim]{Centre for Mathematics and its Applications, The Australian National University, ACT 0200, Australia}
\email{seick.kim@maths.anu.edu.au}
\thanks{Seick Kim is supported by the Australian Research Council.}



\begin{abstract}
We study Green's matrices for divergence form, second order strongly elliptic
systems with bounded measurable coefficients in two dimensional domains.
We establish existence, uniqueness, and pointwise estimates of
the Green's matrices.
\end{abstract}

\maketitle

\mysection{Introduction} \label{intro}
In this article, we study Green's matrices for divergence form,
second order strongly elliptic systems with bounded measurable coefficients
in two dimensional domains.
More precisely, we are concerned with the Green's matrix for elliptic systems
\begin{equation*}
\sum_{j=1}^N L_{ij} u^j\coloneqq \sum_{i=1}^N\sum_{\alpha,\beta=1}^2
D_\alpha(A^{\alpha\beta}_{ij}(x) D_\beta u^j), \quad i=1,\ldots,N
\end{equation*}
in an open connected set $\Omega\subset \bR^2$. Here, $A^{\alpha\beta}_{ij}(x)$
are bounded measurable functions on $\Omega$ satisfying the strong ellipticity
condition.
By a Green's matrix we mean an $N\times N$ matrix valued function
$\vec G(x,y)=(G_{ij}(x,y))_{i,j=1}^N$ defined on 
$\{(x,y)\in\Omega\times\Omega: x\neq y\}$
satisfying the following properties
(see Theorem~\ref{thm1} below for more precise statement):
\begin{align*}
&\sum_{j=1}^N L_{ij} G_{jk}(\cdot,y)=-\delta_{ik}\delta_y (\cdot)
\quad \forall y\in\Omega,\\
&G_{ij}(\cdot,y)=0\quad\text{on}\quad
\partial\Omega \quad \forall y\in\Omega,
\end{align*}
where $\delta_{ik}$ is the Kronecker delta symbol and
$\delta_{y}(\cdot)$ is the Dirac delta function with a unit mass at $y$.
In the scalar case (i.e., when $N=1$), the Green's matrix becomes a real
valued function and is usually called the Green's function.

We prove that if $\Omega$ has either finite volume or finite width, then 
there exists a unique Green's matrix in $\Omega$; see Theorem~\ref{thm1}.
The same is true when $\Omega$ is a domain above a Lipschitz
graph (e.g., $\Omega=\bR^2_+$); see Theorem~\ref{thm2}.
We also establish growth properties of the Green's matrices
including logarithmic pointwise bounds.
We emphasize that we do not require $\Omega$ to be bounded nor to have
a regular boundary in Theorem~\ref{thm1}. Compared to the result of
Dolzmann and M\"uller \cite{DM}, where $\Omega$ is assumed
to be a bounded Lipschitz domain, our result is an improvement
in this respect.
Although there is no Green's matrix for $\Omega=\bR^2$,
there is a possible definition of a fundamental matrix in the entire plane.
Such a construction was carried out by Kenig and Ni \cite{KN} in the scalar
case and by Auscher, McIntosh, and Tchamitchian \cite{AMcT}
in the systems setting
(In fact, Auscher et al. considered complex coefficients elliptic equations
in \cite{AMcT}, but with appropriate changes their strategy carries over
to more general elliptic systems).
For the completeness of presentation, we include the result of Auscher et al.
\cite{AMcT} in Section~\ref{sec5}. 

Let us briefly review the history of works in this area.
In the scalar case, the basic facts about Green's functions of
symmetric elliptic operators in bounded domains were proved by
Littman, Stampacchia, and Weinberger \cite{LSU}.
The study of the Green's functions for nonsymmetric elliptic operators
in bounded domains $\Omega\subset \bR^n$ ($n\geq 3$) was carried out
by Gr\"{u}ter and Widman \cite{GW}.
As it is mentioned earlier, there is no Green's function for $\Omega=\bR^2$;
the fundamental solution $-(1/2\pi)\ln |x-y|$ of Laplace equation
changes sign and is not considered as a Green's function from a point
of view of the classical potential theory (see e.g., \cite{Doob}).
Nevertheless, it is still possible to define a fundamental solution in $\bR^2$.
By using the maximum principle, Kenig and Ni \cite{KN} constructed one
for symmetric elliptic operators.
In \cite{CL}, Chanillo and Li derived that the fundamental solution constructed
by Kenig and Ni is a function of bounded mean oscillation in $\bR^2$.
Also, we would like to bring attention to a paper by Escauriaza
\cite{Escauriaza} on the fundamental solutions of elliptic and parabolic
equations in nondivergence form.
In the systems setting, the Green's matrices of the elliptic systems
with continuous coefficients in bounded $C^1$ domains have been discussed by
Fuchs \cite{Fuchs}  and Dolzmann and M\"uller \cite{DM}.
In fact, Dolzmann and M\"uller improved the strategy of Fuchs and showed
the existence and pointwise estimate for Green's matrix in bounded Lipschitz
domains $\Omega\subset\bR^2$ without imposing any regularity on the
coefficients.
Recently, Hofmann and Kim \cite{HofKim1} gave a unified approach in
studying Green's functions/matrices in arbitrary domains $\Omega\subset \bR^n$
($n\geq 3$) valid for both scalar equations and systems of elliptic type by
considering a class of operators $L$ such that weak solutions of
$L\vec u=0$ satisfy an interior H\"older estimate.
However, like the method used in Gr\"uter and Widman \cite{GW}, the method of
Hofmann and Kim heavily relied on the assumption
that $n\geq 3$ and could not be applied to the two dimensional case.
An parabolic extension of the result by Hofmann and Kim
was carried out in a very recent paper by Cho, Dong, and Kim \cite{CDK}.
In particular, Cho et al. proved that so called
``Dirichlet heat kernel'' of a strongly elliptic system exists
in any domain $\Omega\subset \bR^2$ (see Corollary~2.9 in \cite{CDK}).
In fact, our basic strategy is to make use of their result and
construct the Green's matrix out of the ``Dirichlet heat kernel'' by
integrating in $t$-variable.

The organization of this paper is as follows.
In Section~\ref{main}, we introduce some notations and then state our main
results, Theorem~\ref{thm1} and Theorem~\ref{thm2}.
We give the proof of Theorem~\ref{thm1} in Section~\ref{sec3} and 
that of Theorem~\ref{thm2} in Section~\ref{sec4}.
Finally, in Section~\ref{sec5} we introduce the result of Auscher et al.
\cite{AMcT} regarding construction of a fundamental matrix for an elliptic
system in the entire plane.

\mysection{Preliminaries and main results} \label{main}

\subsection{Strongly elliptic systems in $\bR^2$} \label{sec2.1}
Throughout this article, the summation convention over repeated indices
shall be assumed.
Let $L$ be a second order elliptic operator of divergence type acting on
vector valued functions
$\vec{u}=(u^1,\ldots,u^N)^T$ ($N\geq 1$) defined on an open set
$\Omega\subset \bR^2$ in the following way:
\begin{equation}
\label{eqP-01}
L\vec{u} = D_\alpha (\vec{A}^{\alpha\beta}\, D_\beta \vec{u})
\quad\left(\coloneqq \sum_{\alpha=1}^2 \sum_{\beta=1}^2
D_\alpha (\vec{A}^{\alpha\beta}\, D_\beta \vec{u})\right),
\end{equation}
where $\vec{A}^{\alpha\beta}=\vec{A}^{\alpha\beta}(x)$
($\alpha,\beta=1,2$)
are $N$ by $N$ matrices satisfying the strong ellipticity condition, i.e., 
there is a number $\lambda>0$ such that 
\begin{equation}
\label{eqP-02}
A^{\alpha\beta}_{ij}(x)\xi^j_\beta\xi^i_\alpha
\ge \lambda \abs{\vec{\xi}}^2
\coloneqq\lambda\sum_{i=1}^N\sum_{\alpha=1}^2|\xi^i_\alpha|^2
\quad\forall x\in\Omega.
\end{equation}
We also assume that $A^{\alpha\beta}_{ij}$ are bounded, i.e.,
there is a number $\Lambda>0$ such that 
\begin{equation}
\label{eqP-03}
\sum_{i,j=1}^N\sum_{\alpha,\beta=1}^2
|A^{\alpha\beta}_{ij}(x)|^2\le \Lambda^2\quad\forall x\in\Omega.
\end{equation}
We do not impose any further condition other than \eqref{eqP-02}
and \eqref{eqP-03} on the coefficients.
Especially, we do not assume the symmetry of the coefficients.
The transpose operator ${}^t\!L$ of $L$ is defined by
\begin{equation}
\label{eqP-05}
{}^t\!L\vec{u} = D_\alpha ({}^t\!\vec{A}^{\alpha\beta} D_\beta \vec{u}),
\end{equation}
where ${}^t\!\vec{A}^{\alpha\beta}=(\vec{A}^{\beta\alpha})^T$ 
(i.e., ${}^t\!A^{\alpha\beta}_{ij}=A^{\beta\alpha}_{ji}$).
Note that the coefficients ${}^t\!A^{\alpha\beta}_{ij}$ satisfy
the conditions \eqref{eqP-02}, \eqref{eqP-03} with the same constants $\lambda, \Lambda$.
\subsection{The function space $Y^{1,2}_0(\Omega)$}
The function space $Y^{1,2}_0(\Omega)$ is defined as the set of all weakly
differentiable functions on $\Omega$ such that $D u\in L^2(\Omega)$ and
$u \eta \in W^{1,2}_0(\Omega)$ for any $\eta\in C^\infty_c(\bR^2)$.
An open set $\Omega\subset \bR^2$ is said to be a Green domain if
$\{u 1_\Omega: u\in C^\infty_c(\bR^2)\}\nsubset W^{1,2}_0(\Omega)$.
We ask the readers to refer \cite[\S 1.3.4]{MZ}
for the proofs of lemmas stated below.
\begin{lemma}
\label{lem:P02}
Let $\Omega\subset \bR^2$ be a Green domain and $B\subset\bR^2$ be a ball.
Then, there is a constant $C=C(\Omega,B)$ such that
\begin{equation}
\label{eq:P178}
\norm{u}_{L^2(\Omega\cap B)}\leq C\norm{D u}_{L^2(\Omega)}\quad
\forall u\in Y^{1,2}_0(\Omega).
\end{equation}
\end{lemma}

\begin{lemma}
\label{lem:P03}
Let $\Omega\subset \bR^2$ be a Green domain.
Then $Y^{1,2}_0(\Omega)$ is a Hilbert space when endowed with the inner product
\begin{equation}
\label{eq:P179}
\ip{u,v}\coloneqq\int_\Omega D_i u D_i v.
\end{equation}
\end{lemma}

\begin{lemma}
\label{lem:P04}
Let $\Omega\subset \bR^2$ be a Green domain.
Then $C^\infty_c(\Omega)$ is a dense subset of the Hilbert
space $Y^{1,2}_0(\Omega)$ equipped with the inner product \eqref{eq:P179}.
\end{lemma}

For a given function $\vec f=(f^1,\ldots,f^N)^T\in L^1_{loc}(\Omega)^N$,
we shall say that
$\vec u=(u^1,\ldots,u^N)^T$ is a weak solution in $Y^{1,2}_0(\Omega)^N$
of $L\vec u=-\vec f$ if $\vec u \in Y^{1,2}_0(\Omega)^N$ and
\begin{equation}
\label{eq:P08c}
\int_\Omega A^{\alpha\beta}_{ij} D_\beta u^j D_\alpha \phi^i=
\int_\Omega f^i\phi^i
\quad \forall \vec \phi\in C^\infty_c(\Omega)^N.
\end{equation}
It is routine to check that if $\Omega$ is a Green domain and
$\vec u$ is a weak solution in $Y^{1,2}_0(\Omega)^N$ of $L\vec u=0$, then
$\vec u\equiv 0$. Therefore, a weak solution in $Y^{1,2}_0(\Omega)^N$ of
$L\vec u = -\vec f$ is unique.

\subsection{Main results}
Let us state our main results.
First, we consider domains with either finite volume or finite width.
We shall denote by $|\Omega|$ the Lebesgue measure of $\Omega$
and by $\delta(\Omega)$ the width of $\Omega\subset\bR^2$;
more precisely, we define
\begin{equation}
\delta(\Omega)\coloneqq\inf\set{\dist(\ell_1,\ell_2):\Omega
\text{ lies between two parallel lines }\ell_1,\ell_2};\quad
\inf \emptyset=\infty.
\end{equation}
\begin{theorem} \label{thm1}
Let the operator $L$ satisfy the conditions \eqref{eqP-02} and \eqref{eqP-03}.
Assume that $\Omega\subset \bR^2$ is an open connected set with either
finite volume or finite width so that
\begin{equation}
\label{eq001}
\gamma=\gamma(\Omega)\coloneqq\max\big(|\Omega|^{-1},\delta(\Omega)^{-2}\big)>0.
\end{equation}
Then, there exists a Green's matrix
$\vec{G}(x,y)=(G_{ij}(x,y))_{i,j=1}^N$ defined on
$\{(x,y)\in\Omega\times\Omega:x\neq y\}$ satisfying the properties that
\begin{equation}
\label{eqG-54}
\int_{\Omega}A^{\alpha\beta}_{ij} D_\beta G_{jk}(\cdot,y)
D_\alpha \phi^i = \phi^k(y)
\quad \forall \vec{\phi}\in C^\infty_c(\Omega)^N
\end{equation}
and that for all $\eta\in C^\infty_c(\Omega)$ satisfying
$\eta\equiv 1$ on $B_r(y)$ for some $r<d_y$,
\begin{equation}
\label{eqG-55}
(1-\eta)\vec{G}(\cdot,y)\in Y^{1,2}_0(\Omega)^{N\times N}.
\end{equation}
The Green's matrix $\vec G(x,y)$ in $\Omega$ is unique in the following sense:
\begin{enumerate}
\item
$\vec G(x,y)$ is continuous in $\set{(x,y)\in\Omega\times\Omega:x\neq y}$.
\item
$\vec G(x,\cdot)$ is locally integrable for all $x\in\Omega$.
\item For any
$\vec{f}=(f^1,\ldots,f^N)^T \in C^\infty_c(\Omega)^N$,
the function $\vec u=(u^1,\ldots,u^N)^T$ given by
\begin{equation}
\label{eqZ-70}
\vec{u}(x)\coloneqq\int_{\Omega} \vec{G}(x,y)\vec{f}(y)\,dy \quad
\left(i.e.,\, u^i(x)\coloneqq\int_{\Omega} G_{ij}(x,y) f^j(y)\,dy\right)
\end{equation}
is a unique weak solution in $Y^{1,2}_0(\Omega)^N$ of $L\vec{u}=-\vec{f}$.
\end{enumerate}
Moreover, $\vec G(x,y)$ satisfies the following pointwise estimate:
\begin{equation}
\label{eqG-56}
|\vec G(x,y)| \leq C\left(\frac{1}{\gamma R^2}+\ln \frac{R}{|x-y|}\right)
\quad\text{if } |x-y|<R\coloneqq\tfrac{1}{2}\max(d_x,d_y),
\end{equation}
where $d_x\coloneqq\dist(x,\partial\Omega)$ and
$C=C(\lambda,\Lambda,N)<\infty$.
Consequently, $\vec G(\cdot,y)$ and $\vec G(x,\cdot)$
belong to $L^p(B_r(y)\cap\Omega)$ and $L^p(B_r(x)\cap\Omega)$, respectively,
for all $r>0$ and $p\in [1,\infty)$.
Furthermore, $D \vec G(\cdot,y)$ and $D \vec G(x,\cdot)$
belong to $L^p(B_r(y)\cap\Omega)$ and $L^p(B_r(x)\cap\Omega)$, respectively,
for all $r>0$ and $p\in [1,2)$.
Finally, we have the following symmetry relation:
\begin{equation}
\label{eq:E23}
\vec G(y,x)= {}^t\vec G(x,y)^T\qquad
(\,\text{i.e, }\,G_{ij}(y,x)={}^tG_{ji}(x,y)\,),
\end{equation}
where ${}^t\vec G(x,y)$ is the Green's matrix of the transpose operator
${}^tL$ in $\Omega$.
\end{theorem}

\begin{remark}
When $|\Omega|<\infty$, we have global $L^p$ estimates
for the Green's matrix and its derivatives.
In that case, it will be evident from the proof of Theorem~\ref{thm1} that
$\vec G(\cdot,y)$ and $\vec G(x,\cdot)$ belong to $L^p(\Omega)$
for all $p\in [1,\infty)$ and that $D \vec G(\cdot,y)$ and $D \vec G(x,\cdot)$
belong to $L^p(\Omega)$ for all $p\in [1,2)$.
\end{remark}
Next, we consider a domain above a Lipschitz graph.
Let $\Omega$ be given by
\begin{equation}
\label{eq:LD01}
\Omega=\{(x_1,x_2)\in \bR^2: x_2>\varphi(x_1)\},
\end{equation}
where $\varphi:\bR\to \bR$ is a Lipschitz function with a Lipschitz
constant $M\coloneqq \norm{\varphi'}_\infty<\infty$.

\begin{theorem} \label{thm2}
Let the operator $L$ satisfy the conditions \eqref{eqP-02} and \eqref{eqP-03}.
Assume that $\Omega$ is given by \eqref{eq:LD01}.
Then, there exists a unique Green's matrix
$\vec{G}(x,y)$ satisfying all the properties of Theorem~\ref{thm1} except
\eqref{eqG-56}.
Instead of \eqref{eqG-56} of Theorem~\ref{thm1}, we have
\begin{equation}
\label{eq:H01j}
|\vec G(x,y)|
\leq C\min\left\{1+\ln_+(d_{x,y}/|x-y|),\,d_{x,y}^\mu |x-y|^{-\mu}\right\}
\quad\forall x,y\in\Omega,\quad x\neq y,
\end{equation}
where $d_{x,y}\coloneqq\min(d_x,d_y)$, $d_x\coloneqq \dist(x,\partial\Omega)$,
$\ln_+t\coloneqq\max(\ln t,0)$, $C=C(\lambda,\Lambda,N,M)<\infty$, and
$\mu=\mu(\lambda,\Lambda,M)\in (0,1)$.
In particular, \eqref{eq:H01j} implies $\vec G(x,y)\to 0$ as $|x-y|\to \infty$.
\end{theorem}
\mysection{Proof of Theorem~\ref{thm1}} \label{sec3}
Throughout this section, we employ the letter $C$ to denote a constant depending
on $\lambda, \Lambda, N$ while we use $C(\alpha,\beta,\ldots)$ to denote
a constant depending on quantities $\alpha,\beta,\ldots,$ as well as
$\lambda, \Lambda, N$.
It should be understood that $C$ may vary from line to line.

Let us recall the following version of Poincar\'e inequality
(see e.g., \cite{GT} for the proof).

\begin{lemma}
\label{lem:poincare}
If $\Omega\subset \bR^2$ is an open connected set with either finite volume
or finite width.
Let $\gamma$ be given as in \eqref{eq001} of Theorem~\ref{thm1}.
Then
\begin{equation}
\label{eq002v}
\norm{u}_{L^2(\Omega)}\leq
(2\gamma)^{-1/2}\norm{D u}_{L^2(\Omega)}
\quad \forall u\in W^{1,2}_0(\Omega).
\end{equation}
\end{lemma}

By using the above lemma, one can show that if $\Omega$ has either
finite volume or finite width, then $\Omega$ is a Green domain and
$Y^{1,2}_0(\Omega)=W^{1,2}_0(\Omega)$
(see e.g. \cite[\S 1.3.4]{MZ}).
In the rest of this section we shall identify $Y^{1,2}_0(\Omega)$ with
$W^{1,2}_0(\Omega)$.

\subsection{Construction of the Green's matrix}
Let $\vec \Gamma(t,x,s,y)$
($x,y\in \Omega$ and $t,s\in \bR$) be
the parabolic Green's matrix given as in \cite[Corollary~2.9]{CDK}.
Note that we have $\vec\Gamma(t,x,s,y)=\vec\Gamma(t-s,x,0,y)$.
Throughout the paper, we shall denote
\begin{align}
\label{eq:k01}
\vec K(t,x,y)&\coloneqq\vec \Gamma(t,x,0,y),\\
\label{eq:k02}
\bar{\vec K}(t,x,y)&\coloneqq\int_0^t\vec K(s,x,y)\,ds.
\end{align}
We record here some properties of $\vec K(t,x,y)$ derived in \cite[Corollary~2.9]{CDK}
for the reference.
Recall that $d_x\coloneqq\dist(x,\partial\Omega)$ for $x\in \Omega$.
\begin{align}
\label{eq5.15c}
&\sup_{t\in (r^2,\infty)} \int_\Omega |\vec K(t,x,y)|^2 \,dx
\leq Cr^{-2} \quad \forall r<d_y,\\
\label{eq5.18}
&\iint_{(0,\infty)\times\Omega\setminus (0,r^2)\times B_r(y)}
|\vec K(t,x,y)|^4 \,dx\,dt \leq Cr^{-4} \quad \forall r<d_y,\\
\label{eq5.15b}
&\iint_{(0,\infty)\times\Omega\setminus (0,r^2)\times B_r(y)}
|D_x\vec K(t,x,y)|^2 \,dx\,dt \leq Cr^{-2} \quad \forall r<d_y,\\
\label{eq5.15}
&\iint_{(0,r^2)\times B_r(y)} |\vec K(t,x,y)|^p \,dx\,dt
\leq C(p) r^{-2p+4} \quad \forall r<d_y \quad\forall p\in [1,2),\\
\label{eq5.19}
&\iint_{(0,r^2)\times B_r(y)} |D_x\vec K(t,x,y)|^p \,dx\,dt \leq C(p) r^{-3p+4}
\quad \forall r<d_y \quad\forall p\in [1,4/3),\\
\label{eq5.18a}
&|\vec K(t,x,y)| \leq C\left\{\max\big(\sqrt{|t|},|x-y|\big)\right\}^{-2}
\quad \text{if } \max\big(\sqrt{|t|},|x-y|\big)<\tfrac{1}{2}\max(d_x,d_y).
\end{align}

We define the Green's matrix $\vec G(x,y)$ as follows:
\begin{equation}
\label{eq:g01}
\vec G(x,y)\coloneqq\lim_{t\to \infty} \bar{\vec K}(t,x,y)
=\int_0^\infty \vec K(s,x,y)\,ds
\quad \forall x,y \in \Omega,\quad x \neq y.
\end{equation}
The next lemma will show that $\vec G(x,y)$ is well defined.

\begin{lemma}
\label{lem:E01}
For any $x,y\in \Omega$ with $x\neq y$, we have
$\int_0^\infty \abs{\vec K(s,x,y)}\,ds<\infty$.
\end{lemma}
\begin{proof}
By \cite[Theorem~2.7]{CDK}, we know that
$t\mapsto \vec K(t,x,y)$ is continuous in $t\in \bR$ for $x\neq y$.
Therefore, we only need to show that 
$\int_a^\infty \abs{\vec K(t,x,y)}\,dt<\infty$ for some $a>0$.
Let $\vec u$ be the $k$-th column of $\vec K(\cdot,\cdot,y)$.
Then, by the local boundedness estimate (see \cite{Kim})
\begin{equation}
\label{eq:E13}
\abs{\vec u(t,x)}\le C\left(\fint_{t-\rho^2}^t\fint_{B_\rho(x)}
\abs{\vec u(s,y)}^2\,dy\,ds\right)^{1/2}\quad\forall t>\rho^2\
\quad\forall\rho<d_x.
\end{equation}
By \eqref{eq002v} of Lemma~\ref{lem:poincare}, 
$I(t)\coloneqq\int_\Omega \abs{\vec u(t,\cdot)}^2$ satisfies
\begin{equation*}
I'(t)= -2\int_\Omega A^{\alpha\beta}_{ij}D_\beta u^j D_\alpha u^i(t,\cdot)
\le -2\lambda \int_\Omega \abs{D \vec u(t,\cdot)}^2
\le -4\lambda\gamma I(t)\quad \forall t>0.
\end{equation*}
Therefore, by using \eqref{eq5.15c} we obtain
\begin{equation}
\label{eq:E15}
\int_\Omega \abs{\vec K(t,\cdot,y)}^2
\leq C e^{-4\lambda\gamma(t-r^2)} r^{-2}\quad\forall t>r^2 \quad \forall r<d_y.
\end{equation}
By combining \eqref{eq:E13} and \eqref{eq:E15}
we have
\begin{equation}
\label{eq:E15h}
\abs{\vec K(t,x,y)}\le Ae^{-2\lambda\gamma t}
\quad \forall t> a,
\end{equation}
where $A=A(d_x,d_y)<\infty$ and $a=a(d_x,d_y)<\infty$. The lemma is proved.
\end{proof}

Next, we show that $\vec G(\cdot,y)$ is continuous in
$\Omega\setminus\set{y}$ for any $y\in\Omega$.
We need the following lemma the proof of which can be found in
\cite[Theorem~3.3]{Kim} (c.f. (2.20) and (2.21) in \cite{CDK}).

\begin{lemma}
\label{lem:P01}
Let $L$ satisfy \eqref{eqP-02} and \eqref{eqP-03}.
If $\vec u(t,x)$ is a weak solution of $\vec u_t-L\vec u=0$ in
$Q_{2r}^{-}\coloneqq Q_{2r}^{-}\big((t_0,x_0)\big)$
($\coloneqq (t_0-4r^2,t_0)\times B_{2r}(x_0)$),
then for all $x,x'\in B_r(x_0)$ and $t\in (t_0-r^2,t_0)$,
\begin{align*}
|\vec u(x,t)-\vec u(x',t)|&\le C |x-x'|^\mu r^{-(1+\mu)}
\norm{D_x\vec u}_{L^2(Q_{2r}^{-})},\\
|\vec u(x,t)-\vec u(x',t)|&\le C |x-x'|^\mu r^{-(2+\mu)}
\norm{\vec u}_{L^2(Q_{2r}^{-})},
\end{align*}
where $\mu=\mu(\lambda,\Lambda)\in (0,1)$.
\end{lemma}

Let $\vec u(t,x)$ be the $k$-th column of $\vec K(t,x,y)$.
Fix $x_0\in \Omega$ with $x_0\neq y$ and choose $r>0$ such that
$r<d_y$ and $B_{2r}(x_0)\subset \Omega\setminus B_r(y)$.
By \cite[Therem~2.7]{CDK} and \eqref{eq:k01}, we find that
$\vec u(t,x)$ is a weak solution of $\vec u_t-L \vec u=0$
in $Q_{2r}^-\big((t_0,x_0)\big)$ for any $t_0\in\bR$.
Therefore, by using Lemma~\ref{lem:P01}, \eqref{eq5.15c}, \eqref{eq5.15b},
and \eqref{eq:E15} we have (recall that $\vec K(t,x,y)\equiv 0$ for $t<0$)
\begin{align}
\label{eq:E21a}
|\vec K(t,x,y)-\vec K(t,x_0,y)|&\le C |x-x_0|^\mu r^{-(2+\mu)}
\quad\forall x\in B_r(x_0) \quad \forall t\in \bR,\\
\label{eq:E21b}
|\vec K(t,x,y)-\vec K(t,x_0,y)|&\le C |x-x_0|^\mu r^{-(2+\mu)}
e^{-2\lambda\gamma(t-r^2)}\quad\forall x\in B_r(x_0)\quad \forall t> 5r^2.
\end{align}
Then, for any $x\in B_r(x_0)$, we have
\begin{align}
\nonumber
|\vec G(x,y)-\vec G(x_0,y)|&\le\int_0^{5r^2} |\vec K(t,x,y)-\vec K(t,x_0,y)|\,dt
+\int_{5r^2}^\infty |\vec K(t,x,y)-\vec K(t,x_0,y)|\,dt\\
\label{eq:E22}
&\le C|x-x_0|^\mu r^{-\mu}(1+r^{-2}\gamma^{-1}). 
\end{align}
Therefore, we find that $\vec G(\cdot,y)$ is locally H\"older
continuous in $\Omega\setminus\set{y}$.
Let ${}^t\vec G(x,y)$ be the Green's matrix of the transpose
operator ${}^tL$ in $\Omega$, i.e.,
\begin{equation}
\label{eq:g02}
{}^t\vec G(x,y)\coloneqq \lim_{t\to \infty} {}^t\bar{\vec K}(t,x,y)
=\int_0^\infty {}^t\vec K(s,x,y)\,ds
\quad \forall x,y \in \Omega,\quad x \neq y,
\end{equation}
where ${}^t \vec K(t,x,y)$ and ${}^t\bar{\vec K}(t,x,y)$ are defined similarly
as in \eqref{eq:k01} and \eqref{eq:k02}.
Let ${}^t\vec \Gamma(s,y,t,x)$ be the parabolic
Green's matrix of $\cLt\coloneqq -\partial_t-{}^tL$ constructed
as in \cite{CDK}. Then by \cite[Lemma~3.5]{CDK}
\begin{equation}
\label{eq:E43t}
{}^t\vec K(t,x,y)
={}^t\vec \Gamma(-t,x,0,y)=\vec \Gamma(0,y,-t,x)^T=\vec K(t,y,x)^T,
\end{equation}
and thus we conclude that
\begin{equation}
\label{eq:E43x}
{}^t\bar{\vec K}(t,x,y) = \bar{\vec K}(t,y,x)^T\quad\text{and}\quad
{}^t\vec G(x,y) = \vec G(y,x)^T.
\end{equation}
In particular, we proved \eqref{eq:E23}.
Since ${}^tL$ satisfies \eqref{eqP-02} and \eqref{eqP-03} with
the same $\lambda, \Lambda$, we find as in \eqref{eq:E22} that
${}^t\vec G(\cdot,x)$ is locally H\"older continuous in
$\Omega\setminus\{x\}$ for all $x\in\Omega$.
Therefore, by \eqref{eq:E23} we conclude that $\vec G(x,y)$ is continuous in
$\set{(x,y)\in\Omega\times\Omega:x\neq y}$.

Next, we prove that $\vec G(x,\cdot)$ is locally integrable for all $x\in\Omega$
and $\vec u$ defined by \eqref{eqZ-70} is a weak solution in
$W^{1,2}_0(\Omega)^N$ of $L\vec u = -\vec f$.

\begin{lemma}
\label{lem:E02}
The following estimates hold uniformly for all $t>0$:
\begin{align}
\label{eq:E25}
&\norm{\bar{\vec K}(t,\cdot,y)}_{L^p(B_\rho(y))}\le C(p) \gamma^{-1}\rho^{2/p-2}
\quad\forall p\in [1,2), \quad \text{where }\rho=d_y.\\
\label{eq:E26}
&\norm{\bar{\vec K}(t,\cdot,y)}_{L^4(\Omega\setminus B_r(y))}
\le C \gamma^{-1}r^{-3/2}
\quad \forall r \leq d_y.\\
\label{eq:E27}
&\norm{D\bar{\vec K}(t,\cdot,y)}_{L^p(B_\rho(y))}\le C(p) \gamma^{-1}\rho^{2/p-3}
\quad \forall p\in [1,4/3),\quad \text{where }\rho=d_y.\\
\label{eq:E28}
&\norm{D\bar{\vec K}(t,\cdot,y)}_{L^2(\Omega\setminus B_r(y))}\le
C \gamma^{-1}r^{-2} \quad \forall r \leq d_y.
\end{align}
\end{lemma}

\begin{proof}
We begin by proving \eqref{eq:E25}.
Fix $p\in [1,2)$. By Minkowski's inequality, we have
\begin{align}
\nonumber
\left(\int_{B_\rho(y)}|\bar{\vec K}(t,x,y)|^p\,dx\right)^{1/p}
&\leq \int_0^t\left(\int_{B_\rho(y)} |\vec K(s,x,y)|^p\,dx\right)^{1/p}ds\\
\label{eq:E29c}
&\leq \int_0^{\rho^2} + \int_{\rho^2}^\infty
\left(\int_{B_\rho(y)} |\vec K(s,x,y)|^p\,dx\right)^{1/p}ds \coloneqq I_1+I_2.
\end{align}
We estimate $I_1$ by using H\"older's inequality and \eqref{eq5.15}
as follows:
\begin{equation}
\label{eq:E31a}
I_1 \leq
\left(\int_0^{\rho^2}\!\!\!\int_{B_\rho(y)}
|\vec K(s,x,y)|^p\,dx\,ds\right)^{1/p} \rho^{2(1-1/p)} \leq C(p) \rho^{2/p}.
\end{equation}
To estimate $I_2$, observe that H\"older's inequality and \eqref{eq:E15} yield
(recall $1\leq p<2$)
\begin{align*}
\left(\int_{B_\rho(y)} |\vec K(s,x,y)|^p\,dx\right)^{1/p}
&\leq \left(\int_{B_\rho(y)} |\vec K(s,x,y)|^2\,dx\right)^{1/2}
\abs{B_\rho(y)}^{1/p-1/2}\\
&\leq C(p) \rho^{2/p-2} e^{-2\lambda\gamma (s-\rho^2)}\quad \forall s>\rho^2.
\end{align*}
Therefore, we obtain
\begin{equation}
\label{eq:E31b}
I_2 \leq C(p) \rho^{2/p-2}
\int_{\rho^2}^\infty e^{-2\lambda\gamma (s-\rho^2)}\,ds
\leq C(p) \rho^{2/p-2}\gamma^{-1}.
\end{equation}
Since $1\leq \rho^{-2}\gamma^{-1}$ in any case,
we obtain \eqref{eq:E25} by combining \eqref{eq:E31a} and \eqref{eq:E31b}.

Next, we prove \eqref{eq:E26}.
By using Minkowski's inequality as in \eqref{eq:E29c}, we have
\begin{equation}
\nonumber
\left(\int_{\Omega\setminus B_r(y)}|\bar{\vec K}(t,x,y)|^4\,dx\right)^{1/4}
\leq \int_0^{r^2}
+ \int_{r^2}^\infty\left(\int_{\Omega\setminus B_r(y)}
|\vec K(s,x,y)|^4\,dx\right)^{1/4}ds
\coloneqq I_3+I_4.
\end{equation}
By proceeding as in \eqref{eq:E31a} but using \eqref{eq5.18} instead, we obtain
\begin{equation}
\label{eq:E31c}
I_3 \leq
\left(\int_0^{r^2}\!\!\!\int_{\Omega\setminus B_r(y)}
|\vec K(s,x,y)|^4\,dx\,ds\right)^{1/4}
r^{2(1-1/4)} \leq C r^{1/2}.
\end{equation}
By a well known embedding theorem (see e.g., \cite[\S II.3]{LSU} or
\cite[Theorem~6.9]{Lieberman}), the energy inequality, and \eqref{eq:E15},
we have for $t>r^2$
\begin{align}
\nonumber
\int_t^\infty\!\!\!\int_{\Omega}|\vec K(s,x,y)|^4\,dx\,ds
&\leq C\left(\sup_{t\leq s}\int_{\Omega}|\vec K(s,x,y)|^2\,dx\right)
\int_t^\infty\!\!\!\int_{\Omega}|D_x\vec K(s,x,y)|^2\,dx\,ds\\
\label{eq:E54a}
&\leq C\left(\int_{\Omega}|\vec K(t,x,y)|^2\,dx\right)^2
\leq C r^{-4} e^{-8\lambda\gamma(t-r^2)}.
\end{align}
Then, by using H\"older's inequality and \eqref{eq:E54a} we estimate
\begin{align}
\nonumber
I_4 &\leq \sum_{j=1}^\infty \left(\int_{jr^2}^{(j+1)r^2}\!\!\!\int_\Omega
|\vec K(s,x,y)|^4\,dx\,ds\right)^{1/4} r^{3/2}
\leq C r^{1/2}\sum_{j=1}^\infty e^{-2\lambda\gamma(j-1)r^2}\\
\label{eq:E54b}
&\leq C r^{1/2}(1+r^{-2}\gamma^{-1}).
\end{align}
By combining \eqref{eq:E31c} and \eqref{eq:E54b}, we get \eqref{eq:E26}.

We now turn to the proof of \eqref{eq:E27}.
Fix $p\in [1,4/3)$.
As in \eqref{eq:E29c}, we have
\begin{equation*}
\left(\int_{B_\rho(y)}|D_x\bar{\vec K}(t,x,y)|^p\,dx\right)^{1/p}
\leq \int_0^{\rho^2}
+ \int_{\rho^2}^\infty\left(\int_{B_\rho(y)}
|D_x\vec K(s,x,y)|^p\,dx\right)^{1/p}ds \coloneqq I_5+I_6.
\end{equation*}
By H\"older's inequality and \eqref{eq5.19}, we find
\begin{equation}
\label{eq:E31g}
I_5 \leq
\left(\int_0^{\rho^2}\!\!\!\int_{B_\rho(y)}
|D_x\vec K(s,x,y)|^p\,dx\,ds\right)^{1/p}
\rho^{2(1-1/p)} \leq C(p) \rho^{-1+2/p}.
\end{equation}
To estimate $I_6$, note that H\"older's inequality implies
(recall $1 \leq p < 4/3$)
\begin{equation}
\label{eq:E71a}
\left(\int_{B_\rho(y)}|D_x\vec K(s,x,y)|^p\right)^{1/p} \leq
\left(\int_{B_\rho(y)}|D_x\vec K(s,x,y)|^2\right)^{1/2}|B_\rho(y)|^{1/p-1/2}.
\end{equation}
As in \eqref{eq:E54a}, the energy inequality and \eqref{eq:E15} yield
\begin{equation}
\label{eq:E71b}
\int_t^\infty\!\!\!\int_{\Omega}|D_x\vec K(s,x,y)|^2\,dx\,ds
\leq C \rho^{-2} e^{-4\lambda\gamma(t-\rho^2)}\quad \forall t>\rho^2.
\end{equation}
Then, as in \eqref{eq:E54b}, we estimate
$I_6$ by combining \eqref{eq:E71a} and \eqref{eq:E71b}
\begin{equation}
\label{eq:E31h}
I_6 \leq C(p) \rho^{2/p}\sum_{j=1}^\infty
\left(\int_{j\rho^2}^{(j+1)\rho^2}\!\!\!\int_\Omega
|D_x\vec K(s,x,y)|^2\,dx\,ds\right)^{1/2}
\leq C(p) \rho^{2/p-1}(1+\rho^{-2}\gamma^{-1}).
\end{equation}
We obtain \eqref{eq:E27} by adding \eqref{eq:E31g} and \eqref{eq:E31h}.

Finally, we prove \eqref{eq:E28}.
By using Minkowski's inequality again, we have
\begin{equation}
\nonumber
\left(\int_{\Omega\setminus B_r(y)}|D_x\bar{\vec K}(t,x,y)|^2\,dx\right)^{1/2}
\leq \int_0^{r^2}
+ \int_{r^2}^\infty\left(\int_{\Omega\setminus B_r(y)}
|D_x \vec K(s,x,y)|^2\,dx\right)^{1/2}ds
\coloneqq I_7+I_8.
\end{equation}
We estimate $I_7$ by using H\"older's inequality and \eqref{eq5.15b}:
\begin{equation}
\label{eq:E31v}
I_7 \leq \left(\int_0^{r^2}\!\!\!\int_{\Omega\setminus B_r(y)}
|D_x\vec K(s,x,y)|^2\,dx\,ds\right)^{1/2} r \leq C.
\end{equation}
Also, by using \eqref{eq:E71b} and proceeding as in \eqref{eq:E31h}, we obtain
\begin{equation}
\label{eq:E29q}
I_8 \leq
r\sum_{j=1}^\infty \left(\int_{jr^2}^{(j+1)r^2}\!\!\!\int_\Omega
|D_x\vec K(s,x,y)|^2\,dx\,ds\right)^{1/2} \leq C(1+r^{-2}\gamma^{-1}).
\end{equation}
Therefore, \eqref{eq:E28} follows from \eqref{eq:E31v} and \eqref{eq:E29q}.
The lemma is proved.
\end{proof}

Fix $p_0\in(1,2)$ and $r\leq d_y$. By \eqref{eq:g01} and  \eqref{eq:E25},
there exists a sequence
$\{t_m\}_{m=1}^\infty$ tending to infinity such that
$\bar{\vec K}(t_m,\cdot,y)\wto \vec G(\cdot,y)$ weakly in $L^{p_0}(B_r(y))$,
and thus we have
\begin{equation*}
\norm{\vec G(\cdot,y)}_{L^{p_0}(B_r(y))}\leq C(p_0,\gamma,d_y,r)<\infty
\quad \forall r\leq d_y.
\end{equation*}
By a similar reasoning, \eqref{eq:E26} yields that
\begin{equation*}
\norm{\vec G(\cdot,y)}_{L^4(\Omega\setminus B_r(y))}\leq
C \gamma^{-1}r^{-3/2} \quad \forall r\in (0,d_y].
\end{equation*}
The above inequalities together with \eqref{eq:E23} imply that
$\vec G(x,\cdot)$ is locally integrable for any $x\in\Omega$.
Therefore, the integral in \eqref{eqZ-70} is absolutely convergent
for any $\vec f\in C^\infty_c(\Omega)$, and thus $\vec u$ is well defined
in \eqref{eqZ-70}.
Moreover, \eqref{eq:E25} and \eqref{eq:E26} together with \eqref{eq:E43x} imply
\begin{equation}
\label{eq:E29x}
\vec v(t,x)\coloneqq \int_\Omega \bar{\vec K}(t,x,y)\vec f(y)\,dy
\end{equation}
is well defined.
By the dominated convergence theorem, we also find that
\begin{equation}
\label{eq:E29a}
\lim_{t\to\infty}\vec v(t,x)=\int_\Omega \vec G(x,y)\vec f(y)\,dy=\vec u(x).
\end{equation}
Also, by the definition of $\bar{\vec K}(t,x,y)$ in \eqref{eq:k02},
it is easy to verify
\begin{equation}
\label{eq:E29i}
\vec v_t(t,x)=\int_\Omega \vec K(t,x,y)\vec f(y)\,dy\quad \forall t>0.
\end{equation}
Then, as in the proof of Lemma~\ref{lem:E01}, we have
\begin{equation}
\label{eq:E29b}
\norm{\vec v_t(t,\cdot)}_{L^2(\Omega)}^2\leq
C e^{-4\lambda\gamma t}\norm{\vec f}_{L^2(\Omega)}^2\quad \forall t>0.
\end{equation}
We need the following lemma to show that $\vec u$ is a weak solution in
$W^{1,2}_0(\Omega)^N$ of $L\vec u=-\vec f$.
The readers are asked to consult \cite{CDK} or \cite{LSU}
for the definition of $\rV^{1,0}_2((0,T)\times\Omega)$, etc.
\begin{lemma}
\label{lem:E02b}
For all $T>0$, the function $\vec v$ defined in \eqref{eq:E29x}
is the unique weak solution in $\rV^{1,0}_2((0,T)\times\Omega)^N$
of the problem
\begin{equation}
\label{eq:E32}
\vec v_t- L\vec v=\vec f,\quad
\vec v=0\,\text{ on }\,(0,T)\times\partial\Omega,\quad \vec v(0,\cdot)= 0.
\end{equation}
\end{lemma}
\begin{proof}
Let $\vec w$ be the weak solution in $\rV^{1,0}_2((0,T)\times\Omega)^N$
of the problem \eqref{eq:E32},
the existence and uniqueness of which can be found in \cite{LSU}.
We only need to show that $\vec v\equiv \vec w$ in $(0,T)\times\Omega$.
Fix $t\in (0,T)$ and $x\in\Omega$.
Let ${}^t\vec{\Gamma}(s,y,t,x)$ be the parabolic Green's matrix of
$\cLt\coloneqq -\partial_t-{}^tL$ constructed as in \cite{CDK}.
Then, by proceeding similarly as in the proof of \cite[Theorem~2.7]{CDK},
we obtain (c.f. \cite[Lemma~3.1]{CDK})
\begin{align*}
w^k(t,x)&=
\int_0^T\!\!\!\int_\Omega {}^t\Gamma_{ik}(s,y,t,x)f^i(y)\,dy\,ds=
\int_0^T\!\!\!\int_\Omega \Gamma_{ki}(t,x,s,y)f^i(y)\,dy\,ds\\
&=\int_0^t\!\!\!\int_\Omega K_{ki}(t-s,x,y)f^i(y)\,dy\,ds
=\int_0^t\!\!\!\int_\Omega K_{ki}(s,x,y)f^i(y)\,dy\,ds
=v^k(t,x),
\end{align*}
where we have used $\vec K(t,x,y)\equiv 0$ for $t<0$.
The lemma is proved.
\end{proof}

Note that \eqref{eq:E29b} particularly implies
$\vec v_t(t,\cdot)\in L^2(\Omega)^N$, and thus it is not hard to verify
\begin{equation}
\label{eq:E35}
\int_\Omega v^i_t(t,\cdot)\phi^i+
\int_\Omega A^{\alpha\beta}_{ij}Dv^j(t,\cdot)D_\alpha\phi^i=
\int_\Omega f^i\phi^i\quad \forall\vec\phi\in W^{1,2}_0(\Omega)^N
\quad\forall t>0.
\end{equation}
Then, by setting $\vec\phi=\vec v(t,\cdot)$ above, we find that
for almost all $t>0$,
\begin{align*}
\lambda\norm{D \vec v(t,\cdot)}_{L^2(\Omega)}^2 &\leq
\left(\norm{\vec f}_{L^2(\Omega)}+\norm{\vec v_t(t,\cdot)}_{L^2(\Omega)}\right)
\norm{\vec v(t,\cdot)}_{L^2(\Omega)} \\
&\leq C\gamma^{-1/2}\norm{\vec f}_{L^2(\Omega)}
\norm{D\vec v(t,\cdot)}_{L^2(\Omega)},
\end{align*}
where we have used \eqref{eq:E29b} and Lemma~\ref{lem:poincare}.
Therefore, for almost all $t>0$,
\begin{equation}
\norm{D \vec v(t,\cdot)}_{L^2(\Omega)} \leq
C\gamma^{-1/2}\norm{\vec f}_{L^2(\Omega)}.
\end{equation}
Then, by the weak compactness and \eqref{eq:E29a}, we find that there exists
an increasing sequence $\set{t_m}_{m=1}^\infty$ tending to infinity such that
\begin{equation}
\label{eq:E36}
\lim_{m\to\infty}\int_\Omega A^{\alpha\beta}_{ij}Dv^j(t_m,\cdot)D_\alpha\phi^i=
\int_\Omega A^{\alpha\beta}_{ij}Du^jD_\alpha\phi^i
\quad \forall\vec\phi\in W^{1,2}_0(\Omega)^N.
\end{equation}
Therefore, it follows from
\eqref{eq:E29b}, \eqref{eq:E35}, and \eqref{eq:E36} that $\vec u$ defined in
\eqref{eqZ-70} is a weak solution in $W^{1,2}_0(\Omega)^N$ of $L\vec u=-\vec f$.

Now, we prove the uniqueness. Suppose that there exists another
matrix valued function $\tilde{\vec G}(x,y)$ such that
$\tilde{\vec G}(x,y)$ is continuous
in $\set{(x,y)\in\Omega\times\Omega:x\neq y}$,
$\tilde{\vec G}(x,\cdot)$ is locally integrable in $\Omega$
for all $x\in\Omega$, and for all
$\vec f \in C^\infty_c(\Omega)^N$, the function
$\tilde{\vec u}(x)\coloneqq \int_\Omega \tilde{\vec G}(x,y)\vec f(y)\,dy$
is a weak solution of $L\tilde{\vec u}= -\vec f$ in $W^{1,2}_0(\Omega)^N$.
Then, the difference $\vec w\coloneqq \vec u-\tilde{\vec u}$ is a weak solution
of $L\vec w=0$ in $W^{1,2}_0(\Omega)^N$, and thus $\vec w\equiv 0$.
Therefore, we have
\begin{equation}
\label{eq:E50w}
\int_\Omega (\vec G-\tilde{\vec G})(x,y)\vec f(y)\,dy=0
\quad \forall \vec f\in C^\infty_c(\Omega)^N.
\end{equation}
We conclude from \eqref{eq:E50w} that
$\vec G(x,\cdot)\equiv \tilde{\vec G}(x,\cdot)$ in $\Omega\setminus\{x\}$
 for all $x\in\Omega$, and thus $\vec G(x,y)=\tilde{\vec G}(x,y)$
for all $x,y\in\Omega$ with $x\neq y$. We have proved the uniqueness.

\subsection{Proof of identities \eqref{eqG-54} and \eqref{eqG-55}}
\label{sec3.2}
Let us first prove \eqref{eqG-54}.
From \eqref{eq:k01}, \eqref{eq:k02} and
the construction of parabolic Green's matrix
$\vec{\Gamma}(t,x,s,y)$ in \cite{CDK}, it follows (c.f. (3.3) in \cite{CDK})
\begin{equation}
\label{eq:E01}
\int_\Omega K_{ki}(t,\cdot,y)\phi^i+
\int_\Omega A^{\alpha\beta}_{ij}D_\beta\bar K_{jk}(t,\cdot,y)
D_\alpha\phi^i=\phi^k(y)\quad\forall \vec\phi\in C^\infty_c(\Omega)^N.
\end{equation}
We note that \eqref{eq:E15} yields
\begin{equation}
\label{eq:E02}
\lim_{t\to\infty}\int_\Omega K_{ki}(t,\cdot,y)\phi^i=0\quad\forall k=1,\ldots,N.
\end{equation}
If we write $\vec\phi=\zeta\vec\phi +(1-\zeta)\vec\phi$, where
$\zeta\in C^\infty_c(B_r(y))$ such that $\zeta\equiv 1$ on $B_{r/2}(y)$, then
\begin{align}
\nonumber
\int_\Omega A^{\alpha\beta}_{ij}D_\beta\bar K_{jk}(t,\cdot,y) D_\alpha\phi^i
&=\int_{B_r(y)} A^{\alpha\beta}_{ij}D_\beta\bar K_{jk}(t,\cdot,y)
D_\alpha(\zeta\phi^i)\\
\label{eq:E03}
&\quad +\int_{\Omega\setminus B_{r/2}(y)}
A^{\alpha\beta}_{ij}D_\beta\bar K_{jk}(t,\cdot,y) D_\alpha((1-\zeta)\phi^i)
\coloneqq I_1(t)+I_2(t).
\end{align}
By Lemma~\ref{lem:E02} and \eqref{eq:g01}, we find that there exists
an increasing sequence $\set{t_m}_{m=1}^\infty$ tending to infinity such that
\begin{align}
\label{eq:E04}
\lim_{m\to\infty} I_1(t_m)
&=\int_{B_r(y)} A^{\alpha\beta}_{ij}D_\beta G_{jk}(\cdot,y)
D_\alpha(\zeta\phi^i),\\
\label{eq:E05}
\lim_{m\to\infty} I_2(t_m)
&=\int_{\Omega\setminus B_{r/2}(y)}
A^{\alpha\beta}_{ij}D_\beta G_{jk}(t,\cdot,y) D_\alpha((1-\zeta)\phi^i).
\end{align}
Therefore, by combining \eqref{eq:E01}--\eqref{eq:E05},
we obtain \eqref{eqG-54}.

Next, we prove \eqref{eqG-55}.
We claim
\begin{equation}
\label{eq:E18}
\norm{(1-\eta)\bar{\vec K}(t,\cdot,y)}_{W^{1,2}(\Omega)}\le C(\eta,\gamma)
<\infty \quad \forall t>0,
\end{equation}
where $\eta\in C^\infty_c(\Omega)$ is such that
$\eta\equiv 1$ on $B_r(y)$ for some $r<d_y$.
Assume for the moment that the claim is true.
Then, by the weak compactness and \eqref{eq:g01}, there exists
an increasing sequence $\{t_m\}_{m=1}^\infty$ tending to infinity such that
\begin{equation*}
(1-\eta)\bar{\vec K}(t_m,\cdot,y)\wto (1-\eta)\vec G(\cdot,y)
\quad\text{weakly in }W^{1,2}(\Omega).
\end{equation*}
On the other hand, by \cite[Theorem~2.7]{CDK}, we find that
$(1-\eta)\bar{\vec K}(t,\cdot,y) \in W^{1,2}_0(\Omega)$ for all $t>0$.
Since $W^{1,2}_0(\Omega)$ is weakly closed in $W^{1,2}(\Omega)$,
we have $(1-\eta)\vec G(\cdot,y)\in W^{1,2}_0(\Omega)$ as desired.
To complete the proof of \eqref{eqG-55}, it remains to prove the claim
\eqref{eq:E18}.
In fact, by Lemma~\ref{lem:poincare}, it is enough to show
\begin{equation}
\label{eq:E20}
\norm{D((1-\eta)\bar{\vec K}(t,\cdot,y))}_{L^2(\Omega)}\le C(\eta)
<\infty \quad \forall t>0.
\end{equation}
Let us prove \eqref{eq:E20}.
Assume that $\eta$ is supported in a ball $B\subset \bR^2$.
Then
\begin{align}
\nonumber
\norm{D\big((1-\eta)\bar{\vec K}(t,\cdot,y)\big)}_{L^2(\Omega)}
&\le \norm{1-\eta}_{L^\infty}
\norm{D\bar{\vec K}(t,\cdot,y)}_{L^2(\Omega\setminus B_r(y))}
+\norm{D\eta}_{L^\infty}
\norm{\bar{\vec K}(t,\cdot,y)}_{L^2(B \setminus B_r(y))}\\
\nonumber
&\le C(\eta) \norm{D\bar{\vec K}(t,\cdot,y)}_{L^2(\Omega\setminus B_r(y))}
+C(\eta)|B|^{1/4}
\norm{\bar{\vec K}(t,\cdot,y)}_{L^4(\Omega\setminus B_r(y))}\\
\label{eq:E21v}
&\leq C(\eta) C(\gamma,r)=C(\eta,\gamma)<\infty\qquad\forall t>0,
\end{align}
where we have used Lemma~\ref{lem:E02} in the last step.
This completes the proof of \eqref{eq:E20}, and thus \eqref{eqG-55} is proved.
\subsection{Proof of logarithmic bound \eqref{eqG-56} and $L^p$ estimates}
\label{sec3.3}
Without loss of generality, we may assume $d_y \geq d_x$;
otherwise, in light of \eqref{eq:E23}, we may exchange the role of $x$ and $y$.
Note that if $|x-y|<R=d_y/2$, then \eqref{eq5.18a} yields
\begin{equation}
\label{eq5.18y}
|\vec K(t,x,y)| \leq C\left\{\max\big(\sqrt{t},|x-y|\big)\right\}^{-2}
\quad \forall t \in (0,R^2).
\end{equation}
On the other hand, if we set $\rho=R/2$
in \eqref{eq:E13} and $r=\sqrt{3}R/2$
in \eqref{eq:E15} (note that $\rho<d_x$, $r<d_y$, and $\rho^2+r^2=R^2$),
then \eqref{eq:E15h} becomes
\begin{equation}
\label{eq:E15j}
|\vec K(t,x,y)|\leq C R^{-2} e^{-2\lambda\gamma(t-R^2)}\quad \forall t>R^2.
\end{equation}
Then, by using \eqref{eq5.18y} and \eqref{eq:E15j},
we obtain (recall $|x-y|<R$)
\begin{align}
\nonumber
|\vec G(x,y)|
&\leq C\left(\int_0^{|x-y|^2}|x-y|^{-2}\,dt+ \int_{|x-y|^2}^{R^2} t^{-1}\,dt
+ \int_{R^2}^\infty R^{-2}e^{-2\lambda\gamma(t-R^2)}\,dt \right)\\
\nonumber
&=C\big(1+2\ln (R/|x-y|)+ R^{-2}(2\lambda\gamma)^{-1}\big)\\
\label{eq:R12p}
&\leq C\big(R^{-2}\gamma^{-1}+ \ln(R/|x-y|)\big).
\end{align}
We have thus proved \eqref{eqG-56}.
We now turn to the proof of local $p$-summability of $\vec G(\cdot,y)$
and $D\vec G(\cdot,y)$.
Note that \eqref{eq:R12p} particularly implies that
\begin{equation}
\label{eq:R12q}
\norm{\vec G(\cdot,y)}_{L^p(B_\rho(y))} \leq C(p,d_y,\gamma)<\infty
\quad \forall \rho\in (0,d_y/2] \quad
\forall p \in[1,\infty).
\end{equation}
We claim that
$|D\vec G(\cdot,y)|\in L^p(B_\rho(y))$ for all $0<\rho<d_y$ and $1\leq p<2$.
Let $\vec u$ be the $k$-th column of $\vec G(\cdot,y)$.
Then, by \eqref{eqG-55}, we have
\begin{equation*}
\vec u \in W^{1,2}(\Omega\setminus B_\rho(y))^N\quad \forall \rho \in (0,d_y)
\end{equation*}
and thus, by \eqref{eqG-54}, we find that $\vec u$ is a weak solution of
$L \vec u = 0$ in $\Omega\setminus B_\rho(y)$ for any $\rho<d_y$.
It follows from \eqref{eqG-56} that there is $r_0=r_0(\gamma,d_y)<1$ and
$C_0=C_0(\gamma,d_y)<\infty$ such that
\begin{equation}
\label{eqG-56b}
|\vec G(x,y)| \leq C_0 \ln(1/|x-y|)
\quad \forall x\in B_{r_0}(y).
\end{equation}
Fix $r<r_0$ and let $\zeta\in C^\infty_c(B_r(y))$ be a cut-off function
satisfying $\zeta\equiv 1$ on $\overline{B}_{r/2}(y)$ and $|D\zeta| \leq C/r$.
Then, by \eqref{eqG-55} we find
\begin{equation}
\label{eqn:V01a}
(1-\zeta)^2\vec u \in W^{1,2}_0(\Omega')^N\quad
\text{where } \Omega'\coloneqq \Omega\setminus\overline{B}_{r/2}(y).
\end{equation}
Since $\vec u$ is a weak solution in $W^{1,2}(\Omega')^N$ of $L\vec u=0$,
by using \eqref{eqn:V01a} we have
\begin{equation*}
0=\int_{\Omega'}(1-\zeta)^2 A^{\alpha\beta}_{ij} D_\beta u^j D_\alpha u^i
-\int_{\Omega'} 2(1-\zeta) A^{\alpha\beta}_{ij}D_\beta u^j D_\alpha\zeta\,u^i.
\end{equation*}
Therefore, by using the bound \eqref{eqG-56b} we estimate
\begin{equation*}
\int_{\Omega\setminus B_r(y)}\abs{D \vec u}^2
\le C r^{-2}\int_{B_r(y)\setminus B_{r/2}(y)}\abs{\vec u}^2
\le C C_0^2(\ln(r/2))^2.
\end{equation*}
Therefore, we have
\begin{equation}
\label{eqE-34}
\int_{\Omega\setminus B_r(y)}\abs{D \vec G(\cdot,y)}^2
\le C C_0^2(\ln(r/2))^2
\quad\forall r<r_0.
\end{equation}
Next, let $A_t=\set{x\in\Omega:\abs{D_x\vec G(x,y)}>t}$ and choose $r=2/t$.
Then by \eqref{eqE-34}
\begin{equation*}
\abs{A_t\setminus B_r(y)}\le t^{-2}
\int_{A_t\setminus B_r(y)}\abs{D\vec G(\cdot,y)}^2 \le C C_0^2 t^{-2}(\ln t)^2
\quad \forall t> 2/r_0.
\end{equation*}
and $\abs{A_t\cap B_r(y)}\leq |B_r(y)|\leq Ct^{-2}$.
Therefore, we conclude that for any $y\in\Omega$, there exist
$C_1=C_1(\gamma,d_y)<\infty$ and $t_0=t_0(\gamma,d_y)>0$ such that
\begin{equation}
\label{eqE-36}
\abs{\set{x\in\Omega:\abs{D_x \vec G(x,y)}>t}}\le
C_1 t^{-2}(\ln t)^2\quad \forall t>t_0.
\end{equation}
From the estimates \eqref{eqE-36}, it follows that
$|D\vec G(\cdot,y)|\in L^p(B_r(y))$ for all $r< d_y$ and for all $p\in [1,2)$
as we shall demonstrate below. Let $r<d_y$ be given and choose $\tau>t_0$.
Note that
\begin{align*}
\int_{B_r(y)} |D\vec G(\cdot,y)|^p&=
\int_{B_r(y)\cap \{\,|D\vec G(\cdot,y)|\le \tau\}} |D\vec G(\cdot,y)|^p+
\int_{B_r(y)\cap \{\,|D\vec G(\cdot,y)|>\tau\}} |D \vec G(\cdot,y)|^p\\
&\le  \tau^p\abs{B_r(y)}+
\int_{\{\,|D\vec G(\cdot,y)|> \tau\}} |D\vec G(\cdot,y)|^p.
\end{align*}
By using \eqref{eqE-36}, we estimate (recall $\tau>t_0$)
\begin{align*}
\int_{\{\,|D\vec G(\cdot,y)|> \tau\}} |D\vec G(\cdot,y)|^p
&=\int_0^\infty p t^{p-1} |\{\abs{D\vec G(\cdot,y)}> \max(t,\tau)\}|\,dt\\
&\le C_1 \tau^{-2}(\ln \tau)^2\int_0^\tau p t^{p-1}\,dt
+ C_1 \int_\tau^\infty p t^{p-3}(\ln t)^2\,dt.
\end{align*}
Note that the above integrals converge if $0<p<2$,
and thus we have shown that
\begin{equation}
\label{eq:R12w}
\int_{B_r(y)}|D\vec G(\cdot,y)|^p \leq C(p,\gamma,d_y,r)<\infty\quad
\forall r\in (0,d_y) \quad \forall p\in [1,2).
\end{equation}
On the other hand, \eqref{eq:E28} yields
\begin{equation}
\label{eq:T28a}
\norm{D\vec G(\cdot,y)}_{L^2(\Omega\setminus B_r(y))}\le
C \gamma^{-1}r^{-2} \quad \forall r \in (0, d_y).
\end{equation}
By combining \eqref{eq:R12w} and \eqref{eq:T28a}, we find
\begin{equation}
\label{eq:T28b}
\norm{D \vec G(\cdot,y)}_{L^p(B_r(y)\cap\Omega)}<C(p,\gamma,d_y,r)<\infty
\quad \forall r>0\quad \forall p\in [1,2).
\end{equation}
Next, for $r\geq d_y/2$, fix a cut-off function
$\zeta\in C^\infty_c(B_{2r}(y)\setminus \overline B_{d_y/4}(y))$
such that $\zeta\equiv 1$ on
$\overline B_r(y)\setminus B_{d_y/2}(y)$ and $|D\zeta|\leq C/d_y$.
By a similar computation as in \eqref{eq:E21v}, we have
\begin{equation*}
\norm{D(\zeta \vec G(\cdot,y))}_{L^2(B_{2r}(y)\cap\Omega)}
\leq C(\gamma,d_y,r)<\infty.
\end{equation*}
Since $\zeta \vec G(\cdot,y)\in W^{1,2}_0(B_{2r}(y)\cap \Omega)$,
the Sobolev inequality yields
\begin{equation}
\label{eq:T29c}
\norm{\vec G(\cdot,y)}_{L^p(B_r(y)\cap\Omega\setminus\overline B_{d_y/2}(y))}
<C(p,\gamma,d_y,r)<\infty \quad \forall p\in [1,\infty).
\end{equation}
Then by combining \eqref{eq:R12q} and \eqref{eq:T29c}, we obtain
\begin{equation}
\label{eq:T30d}
\norm{\vec G(\cdot,y)}_{L^p(B_r(y)\cap\Omega)}<C(p,\gamma,d_y,r)<\infty
\quad \forall r>0\quad \forall p\in [1,\infty).
\end{equation}
Finally, from \eqref{eq:E23}, \eqref{eq:T28b}, and \eqref{eq:T30d},
it follows that $|D\vec G(x,\cdot)|$ belongs to $L^p(B_r(x)\cap\Omega)$
for all $r>0$ and $1\leq p<2$ and that $|\vec G(x,\cdot)|$ belongs to
$L^p(B_r(x)\cap\Omega)$ for all $r>0$ and $1\leq p<\infty$.
This completes the proof of the theorem.

\mysection{Proof of Theorem~\ref{thm2}} \label{sec4}
Throughout this section, we employ the letter $C$ to denote a constant depending
on $\lambda, \Lambda, N$, and $M\coloneqq\norm{\varphi'}_\infty$.
We use $C(\alpha,\beta,\ldots)$ to denote a constant depending on quantities
$\alpha,\beta,\ldots,$ as well as $\lambda, \Lambda, N, M$.

For $x=(x_1,x_2)\in \Omega$, where $\Omega$ is as in \eqref{eq:LD01},
we shall denote $\bar x\coloneqq (x_1,\varphi(x_1))\in \partial\Omega$.
Note that $d_x$ is comparable to $|x-\bar x|$; more precisely, we have
\begin{equation}
\label{eq:H01t}
d_x\leq |x-\bar x| \leq  \sqrt{1+M^2}\,d_x\quad \forall x\in\Omega.
\end{equation}
We shall use the following notations.
\begin{align}
P_r^-(t_0,x_0)&\coloneqq \{(t,x)\in \bR\times\Omega: t_0-r^2<t<t_0,\, |x-x_0|<r\},\\
S_r^-(t_0,x_0)&\coloneqq \bR\times\partial\Omega\,\cap\,\partial P_r^-(t_0,x_0). 
\end{align}

We ask the readers to consult \cite{CDK} or \cite{LSU}
for the definition of the space $V_2$.
\begin{lemma}
\label{lem:H01}
Assume that the operator $L$ satisfy the conditions \eqref{eqP-02}
and \eqref{eqP-03}.
Let $\Omega$ be given as in \eqref{eq:LD01} and let
$\bar x_0\in \partial\Omega$. 
Assume that $\vec u(t,x)$ is a weak solution in $V_2(P_{2R}^-(t_0,\bar x_0))$
of $\vec u_t- L \vec u=0$ and vanishes on $S_{2R}^-(t_0,\bar x_0)$.
Then, for any $y_0\in B_R(\bar x_0)\cap \Omega$, we have
\begin{equation}
\label{eq6.18v}
\int_{P_\rho^-(t_0,y_0)}|D_x \vec u|^2 \leq
C\left(\frac{\rho}{r}\right)^{2+2\mu}
\int_{P_r^-(t_0,y_0)}|D_x \vec u|^2
\quad \forall \rho<r\leq R,
\end{equation}
where $\mu=\mu(\lambda,\Lambda,M)\in (0,1)$.
As a consequence, for all $t\in (t_0-R^2,t_0)$, we have
\begin{equation}
\label{eq6.18w}
|\vec u(t,x)-\vec u(t,x')| \leq C|x-x'|^\mu R^{-(1+\mu)}
\left(\int_{P_{2R}^-(t_0,\bar x_0)}|D_x\vec u|^2\right)^{1/2}
\quad \forall x,x'\in B_R(\bar x_0).
\end{equation}
\end{lemma}
\begin{proof}
Let $\vec v(x)$ be a weak solution in $W^{1,2}(\Omega\cap B_{2R}(\bar x_0))$
of $L\vec v =0$ which vanishes on $\partial\Omega\cap B_{2R}(\bar x_0)$.
Let $y_0\in B_R(\bar x_0)\cap \Omega$.
By a well-known boundary regularity theory for weak solutions
of elliptic systems in two dimensional Lipschitz domains
(see e.g. \cite{Morrey}), we have
\begin{equation}
\label{eq6.18a}
\int_{B_\rho(y_0)\cap \Omega}|D \vec v|^2 \leq
C\left(\frac{\rho}{r}\right)^{2\mu}
\int_{B_r(y_0)\cap \Omega}|D \vec v|^2
\quad \forall \rho<r\leq R,
\end{equation}
where $\mu=\mu(\lambda,\Lambda,M) \in(0,1)$.
By a routine adjustment of an argument in \cite{Kim},
one can deduce \eqref{eq6.18v} and \eqref{eq6.18w} from \eqref{eq6.18a}.
\end{proof}

Let $\Omega$ be given as in \eqref{eq:LD01}.
It is rather tedious but routine to check that the estimate \eqref{eq6.18v}
allows us to treat $\Omega$ as if $\Omega=\bR^2$ in the proof of
\cite[Theorem~2.7]{CDK}.
Consequently, we have the following estimates:
\begin{align}
\label{eq618a}
&|\vec K(t,x,y)| \leq C \{\max(|t|^{1/2},|x-y|)\}^{-2},\\
\label{eq615c}
&\int_\Omega |\vec K(t,x,y)|^2 \,dx \leq Ct^{-1}
\quad \forall t>0,\\
\label{eq615b}
&\iint_{(0,\infty)\times\Omega\setminus (0,r^2)\times B_r(y)}\!\!\!
|D_x\vec K(t,x,y)|^2 \,dx\,dt \leq Cr^{-2} \quad \forall r>0,\\
\label{eq619}
&\iint_{(0,r^2)\times (B_r(y)\cap \Omega)}\!\!\!
|D_x\vec K(t,x,y)|^p \,dx\,dt \leq C(p) r^{-3p+4}
\quad \forall r>0 \quad\forall p\in [1,4/3).
\end{align}

To show the convergence of the integral in \eqref{eq:g01},
we need the following lemma.
\begin{lemma}
\label{lem:H03}
Let $\Omega$ be given as in \eqref{eq:LD01}. 
There exists $\mu=\mu(\lambda,\Lambda,M)\in (0,1)$ such that
\begin{align}
\label{eq:H01r}
\abs{\vec K(t,x,y)}&\leq C d_x^{\mu} \{\max(|t|^{1/2},|x-y|)\}^{-2-\mu}
\quad\forall x,y\in\Omega,\quad x\neq y.
\end{align}
\end{lemma}
\begin{proof}
Denote $r\coloneqq \max(|t|^{1/2},|x-y|)$.
We may assume that $d_x< r/(10\sqrt{1+M^2})$; otherwise, \eqref{eq:H01r} is
an easy consequence of \eqref{eq618a}.
Let $\vec u(t,x)$ be the $k$-th column of $\vec K(t,x,y)$ and set $R=r/4$.
Then by \eqref{eq:H01t}, \eqref{eq6.18w}, and \eqref{eq615b}, we have
\begin{align}
\nonumber
|\vec u(t,x)|&=|\vec u(t,x)-\vec u(t,\bar x)| \leq C |x-\bar x|^\mu R^{-1-\mu}
\left(\iint_{(0,\infty)\times\Omega\setminus
(0,R^2)\times B_R(y)}\!\!\!|D_x \vec u|^2\right)^{1/2}\\
\label{eq:H02u}
&\leq  C |x-\bar x|^\mu R^{-2-\mu} \leq C d_x^\mu r^{-2-\mu}.
\end{align}
We obtain \eqref{eq:H01r} from \eqref{eq:H02u}. The lemma is proved.
\end{proof}
Then,  it follows from \eqref{eq:H01r} that for all $x,y\in\Omega$
with $x\neq y$, we have
\begin{align}
\label{eq:H15a}
\int_0^\infty|\vec K(t,x,y)|\,dt
=\int_0^{|x-y|^2} + \int_{|x-y|^2}^\infty|\vec K(t,x,y)|\,dt
\leq C d_x^\mu |x-y|^{-\mu}<\infty.
\end{align}
Therefore, $\vec G(x,y)$ given the same as in \eqref{eq:g01}
is well defined and satisfies
\begin{equation}
\label{eq:H01k}
\abs{\vec G(x,y)}\leq C d_x^{\mu}|x-y|^{-\mu} \quad\forall x,y\in\Omega,
\quad x\neq y.
\end{equation}
In fact, by using \eqref{eq618a} and \eqref{eq:H01r} together,
we may obtain a better bound
\begin{align}
\nonumber
|\vec G(x,y)|
&\leq C\left(\int_0^{|x-y|^2}|x-y|^{-2}\,dt+ \int_{|x-y|^2}^{d_x^2} t^{-1}\,dt
+ \int_{d_x^2}^\infty d_x^\mu t^{-1-\mu/2}\,dt \right)\\
\label{eq:H01l}
&\leq C\big(1+\ln(d_x/|x-y|)\big)\qquad\text{if }\, |x-y|<d_x.
\end{align}
Then by combining \eqref{eq:H01k} and \eqref{eq:H01l}, we derive
(recall $\ln_+t\coloneqq \max(\ln t,0)$)
\begin{equation}
\label{eq:H01m}
|\vec G(x,y)| \leq C\min\left\{1+\ln_+(d_x/|x-y|),d_x^\mu |x-y|^{-\mu}\right\}
\quad \forall x,y\in\Omega,\quad x\neq y.
\end{equation}
Recall that \eqref{eq:E23} is a consequence of \eqref{eq:E43t}, which
remain valid here.
Therefore, \eqref{eq:H01j} follows from \eqref{eq:E23} and \eqref{eq:H01m}.
Note that \eqref{eq:H01j} implies that for any $r>0$ and $p\in[1,\infty)$
\begin{align}
\label{eq:H01w}
\norm{\vec G(x,\cdot)}_{L^p(B_r(x)\cap \Omega)}\leq C(p,d_x,r)<\infty;
\quad
\norm{\vec G(\cdot,y)}_{L^p(B_r(y)\cap \Omega)}\leq C(p,d_y,r)<\infty.
\end{align}
In particular, we have shown that $\vec G(x,\cdot)$ is locally integrable
for all $x\in\Omega$.
Next, we show that $\vec G(\cdot,y)$ is locally H\"older continuous
in $\Omega\setminus\{y\}$.
Fix $x_0\in \Omega$ with $x_0\neq y$ and choose $r>0$ such that $r<d_y$ and
$B_{2r}(x_0)\subset \Omega\setminus B_r(y)$.
Similarly as in \eqref{eq:H02u}, Lemma~\ref{lem:H01} yields
\begin{equation}
\label{eq:H21b}
|\vec K(t,x,y)-\vec K(t,x_0,y)| \le C |x-x_0|^\mu t^{-(1+\mu/2)}
\quad \forall x\in B_r(x_0)
\quad \forall t> t_0,
\end{equation}
where $t_0\coloneqq 8(r+\sqrt{1+M^2}d_{x_0})^2$.
Notice that \eqref{eq:E21a} still remains true here.
Therefore, by using \eqref{eq:H21b} instead of \eqref{eq:E21b} and 
proceeding as in \eqref{eq:E22} we obtain
\begin{equation}
\label{eq:H22d}
|\vec G(x,y)-\vec G(x_0,y)|\leq C|x-x_0|^\mu
\quad\forall x\in B_r(x_0).
\end{equation}
Then, it follows from \eqref{eq:E23} and \eqref{eq:H22d}
that $\vec G(x,y)$ is continuous in $\set{(x,y)\in\Omega\times\Omega:x\neq y}$.

Now let us prove that $\vec u$ defined as in \eqref{eqZ-70} is a unique
weak solution in $Y^{1,2}_0(\Omega)^N$ of $L\vec u=-\vec f$.
First observe that $\Omega$ is a Green domain.
Let $\vec v(t,x)$ be defined the same as in \eqref{eq:E29x}.
Then as in \eqref{eq:E29a}, we have $\lim_{t\to\infty} \vec v(t,x)=\vec u(x)$.
Also, $\vec v_t(t,x)$ has the same representation as in \eqref{eq:E29i}.
Then, by \eqref{eq615c} and Minkowski's inequality, we have
\begin{equation}
\label{eq:H29y}
\norm{\vec v_t(t,\cdot)}_{L^2(\Omega)}
\leq C t^{-1/2}\norm{\vec f}_{L^1(\Omega)}\quad\forall t>0
\end{equation}
and thus, by Lemma~\ref{lem:E02b}, we estimate
\begin{equation}
\label{eq:H29w}
\norm{\vec v(t,\cdot)}_{L^2(\Omega)}
\leq C \norm{\vec f}_{L^1(\Omega)} \int_0^t s^{-1/2}\,ds\leq C t^{1/2}
\norm{\vec f}_{L^1(\Omega)}\quad\forall t>0.
\end{equation}
Assume that $\vec f$ is supported in a ball $B\subset \bR^2$.
Then by setting $\vec \phi=\vec v(t,\cdot)$ in \eqref{eq:E35}, we get
\begin{align}
\nonumber
\lambda\norm{D\vec v(t,\cdot)}_{L^2(\Omega)}^2 &\leq
\norm{\vec v_t(t,\cdot)}_{L^2(\Omega)} \norm{\vec v(t,\cdot)}_{L^2(\Omega)}+
\norm{\vec f}_{L^2(\Omega\cap B)} \norm{\vec v(t,\cdot)}_{L^2(\Omega\cap B)}\\
\label{eq:H36e}
&\leq C\norm{\vec f}_{L^1(\Omega\cap B)}^2+ C(\Omega,B)\,
\norm{\vec f}_{L^2(\Omega\cap B)} \norm{D\vec v(t,\cdot)}_{L^2(\Omega)},
\end{align}
where we have used \eqref{eq:H29y}, \eqref{eq:H29w}, and \eqref{eq:P178}.
Then by applying Cauchy-Schwarz inequality to \eqref{eq:H36e}, we find
\begin{equation}
\label{eq:H36g}
\norm{D\vec v(t,\cdot)}_{L^2(\Omega)}^2
\leq C(\Omega,B)\left(\norm{\vec f}_{L^1(\Omega\cap B)}^2+
\norm{\vec f}_{L^2(\Omega\cap B)}^2\right)
\leq C(\Omega,B)\norm{\vec f}_{L^2(\Omega)}^2 \quad\forall t>0.
\end{equation}
Therefore, by the weak compactness and \eqref{eq:E29a},
we conclude that there exists an increasing sequence
$\set{t_m}_{m=1}^\infty$ tending to infinity such that
$D\vec v(t_m,\cdot) \wto D\vec u$ weakly in $L^2(\Omega)^N$
so that \eqref{eq:E36} holds.
Also, by \eqref{eq:H36g} we find $\norm{D\vec u}_{L^2(\Omega)}<\infty$.
By using \eqref{eq:P178} and \eqref{eq:H36g}
it is not hard to verify that for any $\zeta\in C^\infty_c(\bR^2)$,
\begin{align*}
\norm{\zeta \vec v(t,\cdot)}_{W^{1,2}(\Omega)}\leq 
C(\zeta,\Omega,B)\norm{\vec f}_{L^2(\Omega)}<\infty \quad\forall t>0.
\end{align*}
Therefore, we conclude that $\zeta \vec u \in W^{1,2}_0(\Omega)^N$
for all $\zeta\in C^\infty_c(\Omega)$ and thus,
$\vec u \in Y^{1,2}_0(\Omega)^N$.
Consequently, it follows from \eqref{eq:E35}, \eqref{eq:E36}, \eqref{eq:H29y},
and Lemma~\ref{lem:P04} that $\vec u$ is a unique weak solution in
$Y^{1,2}_0(\Omega)^N$ of $L\vec u= -\vec f$.

By arguing the same as in the proof of Theorem~\ref{thm1}, we
get the uniqueness of the Green's matrix in $\Omega$.

We need the following lemma to prove \eqref{eqG-54} and \eqref{eqG-55}.
\begin{lemma}
\label{lem:H04}
Let $\Omega$ be given as in \eqref{eq:LD01}.
Then, for all $y\in\Omega$ and for all $t>0$, we have 
\begin{align}
\label{eq:H17n}
&\norm{D \bar{\vec K}(t,\cdot,y)}_{L^p(B_\rho(y)\cap\Omega)}
\leq C(p)\rho^{2/p-1} \quad \forall \rho>0 \quad \forall p\in [1,4/3),\\
\label{eq:H18p}
&\norm{D \bar{\vec K}(t,\cdot,y)}_{L^2(\Omega\setminus B_r(y))}
\leq C(1+d_y^\mu r^{-\mu})\quad \forall r>0.
\end{align}
\end{lemma}
\begin{proof}
We proceed similarly as in the proof of Lemma~\ref{lem:E02}.
Let us begin by proving \eqref{eq:H17n} first.
By Minkowski's inequality, we have
\begin{equation*}
\left(\int_{B_\rho(y)\cap \Omega}\!\!\!
|D_x\bar{\vec K}(t,x,y)|^p dx\right)^{1/p}
\leq \int_0^{\rho^2}+\int_{\rho^2}^\infty
\left(\int_{B_\rho(y)\cap \Omega}\!\!\!|D_x \vec K(s,x,y)|^p dx\right)^{1/p}ds
\coloneqq I_1+I_2.
\end{equation*}
Then, by H\"older's inequality and \eqref{eq619}, we have
\begin{equation}
\label{eq:H15d}
I_1\leq \left(\int_0^{\rho^2}\!\!\!\int_{B_\rho(y)\cap \Omega}
|D_x \vec K(s,x,y)|^p\,dx\,ds\right)^{1/p} \rho^{2(1-1/p)}
\leq C(p) \rho^{-1+2/p}.
\end{equation}
On the other hand, by using H\"older's inequality
\begin{equation}
\label{eq:H16m}
I_2 \leq C(p) \rho^{2/p}\sum_{j=1}^\infty
\left(\int_{j\rho^2}^{(j+1)\rho^2}\!\!\!\int_{B_\rho(y)\cap \Omega}
|D_x \vec K(s,x,y)|^2\,dx\,ds\right)^{1/2}
\coloneqq C(p)\rho^{2/p} \sum_{j=1}^\infty I_{2,j}.
\end{equation}
By setting $r=\sqrt{(j+1)/2}\,\rho$ in \eqref{eq6.18v}
and using \eqref{eq615b}, we estimate
\begin{align}
\nonumber
(I_{2,j})^2 &=\int_{P_\rho^-(2r^2,y)} |D_x \vec K(t,x,y)|^2 dx\,dt
\leq C\left(\frac{\rho}{r}\right)^{2+2\mu}
\int_{P_r^-(2r^2,y)}|D_x \vec K(t,x,y)|^2\,dx\,dt\\
\nonumber
&\leq C\left(\frac{\rho}{r}\right)^{2+2\mu}
\int_{(0,\infty)\times\Omega\setminus(0,r^2)\times B_r(y)}
|D_x \vec K(t,x,y)|^2\,dx\,dt\\
\label{eq:H17b}
&\leq C\rho^{2+2\mu} r^{-4-2\mu}= C \rho^{-2}(j+1)^{-2-\mu}.
\end{align}
Therefore, by combining \eqref{eq:H16m} and \eqref{eq:H17b}, we find
\begin{equation}
\label{eq:H18c}
I_2\leq C(p) \rho^{2/p-1}\sum_{j=1}^\infty (j+1)^{-1-\mu/2}=C(p)\rho^{2/p-1},
\end{equation}
and thus, \eqref{eq:H17n} follows from \eqref{eq:H15d} and \eqref{eq:H18c}.

Next, we turn to the proof of \eqref{eq:H18p}.
As before, Minkowski's inequality yields
\begin{align*}
\left(\int_{\Omega\setminus B_r(y)}\!\!\!
|D_x\bar{\vec K}(t,x,y)|^2 dx\right)^{1/2} &\leq \int_0^{r^2}+\int_{r^2}^\infty
\left(\int_{\Omega\setminus B_r(y)}\!\!\!
|D_x \vec K(s,x,y)|^2 dx\right)^{1/2}ds\\
&\coloneqq I_3+I_4.
\end{align*}
Then, by H\"older's inequality and \eqref{eq615b}, we have
\begin{equation}
\label{eq:H23f}
I_3\leq \left(\int_0^{r^2}\!\!\!\int_{\Omega\setminus B_r(y)}
|D_x \vec K(s,x,y)|^2\,dx\,ds\right)^{1/2} r
\leq C.
\end{equation}
We need the following inequality to estimate $I_4$:
\begin{equation}
\label{eq:H23e}
I_5(t)\coloneqq 
\int_t^\infty\!\!\!\int_\Omega |D_x \vec K(s,x,y)|^2 dx\,ds
\leq C d_y^{2\mu} t^{-1-\mu}\quad \forall r>0\quad \forall t>0.
\end{equation}
Let us momentarily assume that \eqref{eq:H23e} holds and proceed
similarly as in \eqref{eq:H17b} to get
\begin{align}
\nonumber
I_4 \leq
\left(\int_{r^2}^\infty\!\!\!\int_{\Omega\setminus B_r(y)}
\!\!\!|D_x \vec K(s,x,y)|^2 dx\,ds\right)^{1/2}
&\leq \sum_{j=0}^\infty 2^{j/2}r
\left(\int_{2^j r^2}^{2^{j+1} r^2}\!\!\!\int_\Omega
|D_x \vec K(s,x,y)|^2 dx\,ds\right)^{1/2}\\
\label{eq:H24e}
&\leq Cd_y^\mu r^{-\mu} \sum_{j=0}^\infty 2^{-j\mu/2} \leq C d_y^\mu r^{-\mu}.
\end{align}
By combining \eqref{eq:H23f} and \eqref{eq:H24e}, we obtain \eqref{eq:H18p}.
It only remains to prove \eqref{eq:H23e}.
Note that by \eqref{eq:H01r} and \eqref{eq:E43t} we have
\begin{equation}
\label{eq:H31o}
|\vec K(s,x,y)|\leq C d_y^\mu\{\max(s^{1/2},|x-y|)\}^{-2-\mu}\quad
\forall x\neq y\quad\forall s>0.
\end{equation}
Let $\zeta\in C^\infty(\bR)$ be such that $0\leq \zeta\leq 1$,
$\zeta\equiv 1$ on $[t,\infty)$, $\zeta\equiv 0$ on $(-\infty,t/2]$,
and $|\zeta'| \leq 4/t$.
Then, by the energy inequality (see e.g., \cite[\S III.2]{LSU})
and \eqref{eq:H31o}, we have
\begin{align*}
I_5(t)&\leq
\int_0^\infty\!\!\!\int_\Omega \zeta(s)|D_x \vec K(s,x,y)|^2\,dy\,ds
\leq C \int_0^\infty\!\!\!\int_\Omega |\zeta'(s)| |\vec K(s,x,y)|^2\,dy\,ds\\
&\leq Ct^{-1}d_y^{2\mu} \int_{t/2}^t \left(\int_{|x-y|<\sqrt{s}} s^{-2-\mu}\,dy
+\int_{|x-y|\geq \sqrt{s}} |x-y|^{-4-2\mu}\,dy\right)\,ds\\
&\leq Ct^{-1}d_y^{2\mu} \int_{t/2}^t s^{-1-\mu}\,ds \leq C d_y^{2\mu}
t^{-1-\mu}.
\end{align*}
This completes the proof of the lemma.
\end{proof}

We now prove \eqref{eqG-54} and \eqref{eqG-55}.
To prove \eqref{eqG-54}, first recall that \eqref{eq:E01} holds.
By \eqref{eq615c}, we find that \eqref{eq:E02} remains valid.
Assume that $\vec \phi\in C^\infty_c(\Omega)^N$ is supported in 
$B_R(y)\cap\Omega$.
By \eqref{eq:g01} and \eqref{eq:H17n}, we find that there is a sequence
$\{t_m\}_{m=1}^\infty$ tending to infinity such that
\begin{equation*}
D\vec{\bar K}(t_m,\cdot,y)\wto D \vec G(\cdot,y)\quad
\text{weakly in } L^p(B_R(y)\cap\Omega)^{N\times N}\text{ for some }p>1.
\end{equation*}
Therefore, we find
\begin{equation}
\label{eq:H23w}
\lim_{m\to\infty} \int_\Omega A^{\alpha\beta}_{ij}D_\beta
\bar K_{jk}(\cdot,y) D_\alpha \phi^i
=\int_\Omega A^{\alpha\beta}_{ij}D_\beta G_{jk}(\cdot,y) D_\alpha \phi^i.
\end{equation}
By combining \eqref{eq:E01}, \eqref{eq:E02}, and \eqref{eq:H23w}, we obtain
\eqref{eqG-54}.
To prove \eqref{eqG-55}, first observe that \eqref{eq:H18p} yields
\begin{equation}
\label{eq:H35p}
\norm{D\vec G(\cdot,y)}_{L^2(\Omega\setminus B_r(y))} \leq C(d_y,r)<\infty
\quad \forall r>0.
\end{equation}
By using \eqref{eq:H01w} and \eqref{eq:H35p} and proceeding
similarly as in \eqref{eq:E21v}, we obtain
\begin{equation*}
\norm{D((1-\eta)\vec G(\cdot,y))}_{L^2(\Omega)}\leq
C(\eta,d_y)<\infty.
\end{equation*}
It follows from \cite[Theorem~2.7]{CDK} that for any
$\zeta\in C^\infty_c(\bR^2)$, we have
\begin{equation*}
\vec F(t,\cdot)\coloneqq\zeta(1-\eta)\bar{\vec K}(t,\cdot,y)
\in W^{1,2}_0(\Omega)^{N\times N}\quad\forall t>0.
\end{equation*}
Clearly, $\lim_{t\to\infty}\vec F(t,\cdot)=\zeta(1-\eta)\vec G(\cdot,y)$.
Moreover, by utilizing \eqref{eq:E43x}, \eqref{eq:H01r}, and \eqref{eq:H18p},
it is not hard to verify
$\norm{\vec F(t,\cdot)}_{W^{1,2}(\Omega)}\leq C(\zeta,\eta,d_y)<\infty$
for all $t>0$.
Then, by a similar argument as in Section~\ref{sec3.2}, we get 
$\zeta(1-\eta)\vec G(\cdot,y)\in W^{1,2}_0(\Omega)^{N\times N}$.
We have proved \eqref{eqG-55}.

Finally, notice that with \eqref{eqG-55} at hand,
we may proceed similarly as in Section~\ref{sec3.3} to conclude that
that $D\vec G(\cdot,y)$ and $D\vec G(x,\cdot)$ belong to
$L^p(B_r(y)\cap\Omega)$ and $L^p(B_r(x)\cap\Omega)$, respectively,
for all $r>0$ and $p\in [1,2)$.
We have already seen in \eqref{eq:H01w} that
that $\vec G(\cdot,y)$ and $\vec G(x,\cdot)$ belong to
$L^p(B_r(y)\cap\Omega)$ and $L^p(B_r(x)\cap\Omega)$, respectively,
for all $r>0$ and $p\in [1,\infty)$.
This completes the proof of the theorem.

\mysection{Remark on fundamental matrices} \label{sec5}
In this section, we introduce a result of Auscher et al. \cite{AMcT}
regarding construction of a fundamental matrix in $\bR^2$.
Let $\mathcal H^1(\bR^2)$ be the usual Hardy space in $\bR^2$ and
$C_0(\bR^2)$ be the space of continuous functions on $\bR^2$
vanishing at infinity.
For $x,y\in \bR^2$, $x\neq y$, define
\begin{equation}
\vec \Gamma(x,y)\coloneqq
\int_0^1 \vec K(t,x,y)\,dt+ \int_1^\infty (\vec K(t,x,y)-\vec K(t,x,x))\,dt.
\end{equation}
The following theorem appears in \cite{AMcT} as Theorem~3.16, where $L$ is
assumed to be an elliptic operator with complex coefficients.
With appropriate changes, the same proof carries over.
\begin {theorem}[Auscher-McIntosh-Tchamitchian]
Let the operator $L$ satisfy \eqref{eqP-02} and \eqref{eqP-03}.
Then for all $x\in \bR^2$, $\vec\Gamma(x,\cdot) \in BMO$ and
for $\vec f=(f^1,\ldots,f^N)^T\in \mathcal H^1(\bR^2)^N$,
the function defined by
\begin{equation*}
T\vec f(x)\coloneqq \int_{\bR^2} \vec\Gamma(x,y)\vec f(y)\,dy
\end{equation*}
belongs to $C_0(\bR^2)^N$.
The linear map thus defined is continuous from $\mathcal H^1(\bR^2)^N$
into $C_0(\bR^2)^N$.
Moreover, for all $\vec f\in \mathcal H^1(\bR^2)^N$,
$\vec u(x)\coloneqq T \vec f(x)$ satisfies
$\norm{D \vec u}_{L^2(\bR^2)}<\infty$ and is a weak solution of
$L\vec u= -\vec f$ in the sense of \eqref{eq:P08c}.
\end{theorem}


\end{document}